\documentclass[submission]{FPSAC2022}

\usepackage{lipsum}
\usepackage{amsmath, mathtools}
\numberwithin{equation}{section}
\numberwithin{figure}{section}

\usepackage{etoolbox,subcaption}
\usepackage{tikz}
\usepackage{dynkin-diagrams}

\newtheorem{theorem}{Theorem}[section]
\newtheorem{lemma}[theorem]{Lemma}
\newtheorem{proposition}[theorem]{Proposition}
\newtheorem*{p-psp}{Parabolic partial sum property}

\theoremstyle{definition}
\newtheorem{definition}[theorem]{Definition}

\newtheorem{remark}[theorem]{Remark}
\newtheorem{example}[theorem]{Example}
\newtheorem*{note}{Note}

\newcommand{\conv}[0]{\mathrm{conv}}
\newcommand{\height}[0]{\mathrm{ht}}

\newcommand{\weight}[0]{\mathrm{wt}}

\title[Weak faces and a weight formula for weights, via parabolic partial sum property for roots]{Weak faces and a formula for weights\\ of highest weight modules,\\
via parabolic partial sum property for roots}

\author{G. Krishna Teja
}

\address{\addressmark{1}Department of Mathematics, Indian Institute of
Science, Bangalore -- 560012, India
}

\received{\today}

\abstract{
Let $\mathfrak{g}$ be a finite or an affine type Lie algebra over
$\mathbb{C}$ with root system $\Delta$. We show a parabolic
generalization of the partial sum property for $\Delta$, which we term
the parabolic partial sum property. It allows any root $\beta$ involving
(any) fixed subset $S$ of simple roots, to be written as an ordered sum
of roots, each involving exactly one simple root from $S$, with each
partial sum also being a root. We show three applications of this
property to weights of highest weight $\mathfrak{g}$-modules:
(1)~We provide a minimal description for the weights of all
non-integrable simple highest weight $\mathfrak{g}$-modules, refining the
weight formulas shown by Khare [\textit{J.\ Algebra}\ 2016] and
Dhillon--Khare [\textit{Adv.\ Math.}\ 2017].
(2)~We provide a Minkowski difference formula for the weights of an
arbitrary highest weight $\mathfrak{g}$-module.
(3)~We completely classify and show the equivalence of two combinatorial
subsets -- weak faces and 212-closed subsets -- of the weights of all
highest weight $\mathfrak{g}$-modules. These two subsets were introduced
and studied by Chari--Greenstein [\textit{Adv.\ Math.}\ 2009], with
applications to Lie theory including character formulas. We also show
($3'$) a similar equivalence for root systems.
}


\begin{document}

\maketitle

\section{Introduction}

Let $\mathbb{Z}\subset \mathbb{R}\subset \mathbb{C}$ be the set of
integers, real and complex numbers, respectively. Throughout,
$\mathfrak{g}=\mathfrak{g}(A)$ stands for a complex finite-dimensional
simple or an affine Kac--Moody Lie algebra -- accordingly, in short we
say $\mathfrak{g}$ is of finite or affine type -- corresponding to a
Cartan matrix $A$. We fix for $\mathfrak{g}$: a Cartan subalgebra
$\mathfrak{h}$, triangular decomposition
$\mathfrak{n}^-\oplus\mathfrak{h}\oplus\mathfrak{n}^+$, root system
$\Delta\subset \mathfrak{h}^*$, set of simple roots $\Pi=\{\alpha_i\ | \
i\in \mathcal{I}\}$ in $\Delta$ and simple co-roots
$\Pi^{\vee}=\{\alpha_i^{\vee}\ |\ i\in\mathcal{I}\}\subset\mathfrak{h}$.
Here $\mathcal{I}$ is the (fixed) common indexing set for the simple
roots, simple co-roots, rows/columns of $A$, and the set of nodes in the
Dynkin diagram of $\mathfrak{g}$. We say $\Delta$ is of finite or affine
type depending on the type of $\mathfrak{g}$.

This note showcases in finite and affine types, some of the main results
in \cite{MBWWM} and \cite{WFHMRS}, which study the root systems and
weights of arbitrary highest weight modules (see Section~\ref{S3} for the
definition) over $\mathfrak{g}$.
The results stated below were shown more generally over all
Kac--Moody algebras in the above two papers. These are inspired by and
build on the works of Chari, Khare, and their co-authors
\cite{Chari-Greenstein}, \cite{Dolbin}, \cite{KhareAdv},
\cite{KhareJ.Alg}, \cite{Ridenour} etc.

\subsection{Parabolic partial sum property and its applications}

Begin by recalling the well-known property of root systems $\Delta$, the
\textit{partial sum property (PSP)}: every root in $\Delta$ is an ordered
sum of simple roots such that all partial sums are also roots. In this
paper, we give a novel ``parabolic-generalization'', termed the
\textit{parabolic partial sum property} -- or \textit{parabolic PSP} --
which we state formally in the next section. It has many applications to
representation theory and algebraic combinatorics; this note records
three of them. Informally, the parabolic PSP says that given a nonempty
set $S \subseteq \Pi$ of simple roots, every positive root $\beta$
involving some simple roots from $S$ is an ordered sum of positive roots,
such that each root involves exactly one simple root from $S$, and all
partial sums are also roots. Observe, the case $S = \Pi$ is the ``usual''
PSP.

The parabolic PSP was originally a question of Khare, whose motivation
was to obtain a ``minimal description'' for the weights of \textit{all}
non-integrable simple highest weight $\mathfrak{g}$-modules, and more
generally, all parabolic Verma modules -- thus, the label
``parabolic''.

The simple highest weight modules over $\mathfrak{g}$ are crucial in
representation theory, combinatorics, and physics among other areas.
These modules form the building blocks of, and thereby pave ways for the
study of, modules central to Lie theory. Their weights and characters,
particularly those of \textit{integrable} (simple) highest weight
$\mathfrak{g}$-modules, have rich combinatorial properties and numerous
applications. For instance, when $\mathfrak{g}$ is of finite
type $A$, these characters are precisely the Schur polynomials.   

Recall, the simple highest weight $\mathfrak{g}$-modules are indexed by
their highest weights $\lambda\in\mathfrak{h}^*$; correspondingly, let
$L(\lambda)$ denote the simple module. For a weight module $V$ --
meaning, $V$ is a direct sum of its weight spaces/the simultaneous
eigenspaces of $V$ for the action of $\mathfrak{h}$ -- let $\weight V$
denote the set of weights of $V$. 

The weights and characters of the integrable simple modules $L(\lambda)$
-- for all \textit{dominant integral} $\lambda\in\mathfrak{h}^*$ -- were
well understood decades ago. However, for non-integrable simple
$\mathfrak{g}$-modules $L(\lambda)$, even their sets of weights seem to
have not been known until 2016 \cite{KhareJ.Alg}. In this paper, Khare
showed three formulas for the weights of a large class of highest weight
$\mathfrak{g}$-modules, including all modules $L(\lambda)$.
Dhillon--Khare \cite{KhareAdv} extended these formulas to hold over all
Kac--Moody $\mathfrak{g}$. Here we focus on one of these three formulas,
the Minkowski difference formula \eqref{Minkowski difference formula for
wt L(lambda)}. This uses the $\mathbb{Z}_{\geq 0}$-cone of
the positive roots lying outside a subroot system
$\Delta_{J_\lambda}$. Finding minimal generating sets for these cones
would yield a minimal description for the weights of simple highest
weight $\mathfrak{g}$-modules. In this note we prove the parabolic PSP,
thereby solving the minimal description problem.

As a second application, the analysis and understanding of the above
minimal generating sets and the parabolic PSP, were fruitful in showing
several results about weights of arbitrary highest weight
$\mathfrak{g}$-modules in \cite{MBWWM} and \cite{WFHMRS}. Namely, in
\cite{MBWWM} we obtain a Minkowski difference formula for the set of
weights of \textit{every} highest weight $\mathfrak{g}$-module, which
looks similar to those shown by Dhillon--Khare for simple modules. This
note discusses that result, as well as the ones from \cite{WFHMRS}
explained in the next subsection.   

\subsection{Weak faces of root systems and weights}

The parabolic partial sum property and its two applications (above)
occupy the first half of this note. The second half is a third
application: classifying and determining two combinatorial subsets of
$\weight V$ and its convex hull, for arbitrary highest weight
$\mathfrak{g}$-modules $V$. These are the \textit{weak faces} and the
\textit{212-closed subsets} (see \cite{KhareJ.Alg} for the choice of
these names),
 and they arise from the combinatorics of root systems $\Delta$, and of
subsets that maximize linear functionals over $\Delta$. Such subsets
were studied by Chari and her co-authors in \cite{Chari-Greenstein},
\cite{Dolbin}
and used to show several interesting results in representation theory.
These include constructing Koszul algebras, obtaining character formulas
for the specialization at $q=1$ of Kirillov--Reshetikhin modules over
untwisted quantum affine algebras $U_q(\widehat{\mathfrak{g}})$, and
constructing irreducible ad-nilpotent ideals in parabolic subalgebras of
$\mathfrak{g}$. For a detailed overview of these motivations,
see \cite{KhareJ.Alg} and \cite{WFHMRS}.

Khare and Ridenour \cite{Ridenour} extended the results of
\cite{Chari-Greenstein} from maximizer subsets of $\Delta$ to the weights
falling on the faces of \textit{Weyl polytopes}, i.e., the shapes $\conv
(\weight L(\lambda))$ for all dominant integral
$\lambda\in\mathfrak{h}^*$. Even more generally, Khare \cite{KhareJ.Alg}
considered $\conv (\weight V)$ for highest weight modules $V$
with arbitrary highest weight $\lambda \in \mathfrak{h}^*$,
and introduced:

\begin{definition}\label{weak face defiinition}
Let $\mathbb{A}$ be a fixed non-trivial additive subgroup of
$(\mathbb{R},+)$, and $\emptyset\neq Y\subseteq X$ be two subsets of a
real vector space. Define $\mathbb{A}_{\geq 0}:=\mathbb{A}\cap
[0,\infty)$.
\begin{enumerate}
\item $Y$ is said to be a \textbf{\textit{weak}-$\mathbb{A}$-\textit{face of}} $X$ if
		\[
		\begin{rcases*}
		\sum\limits_{i=1}^{n}r_iy_i=\sum\limits_{j=1}^{m}t_jx_j\text{ and }\sum\limits_{i=1}^{n}r_i=\sum\limits_{j=1}^{m}t_j >0\qquad
		\text{for }m,n\in\mathbb{N},\\
		y_i\in Y,\text{ }
		x_j\in X,\text{ }r_i,t_j\in \mathbb{A}_{\geq 0}\text{ }\forall\text{ }1\leq i\leq n,1\leq j\leq m
		\end{rcases*}
		\implies\text{ }x_j\in Y\ \forall t_j\neq 0.
		\]
		By \textit{weak faces} of $X$ we mean the collection of all weak-$\mathbb{A}$-faces of $X$ for all additive subgroups $\{0\}\subsetneqq\mathbb{A}\subseteq (\mathbb{R},+)$.
\item $Y$ is said to be a \textbf{212-\textit{closed subset of}}
\big(or \textbf{212-\textit{closed in}}\big) $X$ if
		\[
		(y_1)+(y_2)=(x_1)+(x_2)\ \  \text{for some }\ y_1,y_2\in Y\text{ and }x_1,x_2\in X\quad\implies\quad x_1,x_2\in Y.
		\]
\end{enumerate}
\end{definition}

\begin{remark}\label{weak face is 212-closed}
For any pair of subsets $Y \subseteq X$ of a real vector space, and all
non-trivial $\mathbb{A} \subseteq (\mathbb{R},+)$, it can be checked by
the definitions that each part below implies the next: (1) $Y$ maximizes
a linear functional $\psi$ on $X$, i.e., $\psi(x)\leq \psi(y)$ for all
$x\in X$ and $y\in Y$. (2)~$Y$ is a weak-$\mathbb{R}$-face of $X$. (3)
$Y$ is a weak-$\mathbb{A}$-face of $X$. (4) $Y$ is 212-closed in $X$.
\end{remark}

For a subset $X\neq\emptyset$ of a real vector space, weak faces of $X$
generalize the ``classical'' faces of the convex hull $\conv (X)$.
Namely, when $\conv (X)$ is a polyhedron, \cite{Ridenour} shows that weak
faces of $X$ are precisely the elements of $X$ that fall on faces of
$\conv(X)$ -- i.e., conditions (1)--(3) in Remark~\ref{weak face is
212-closed} are equivalent. Our goal is to show that the same holds for
the even weaker notion (a priori) of (4) 212-closed subsets, when $X$ is
a distinguished set in Lie theory and algebraic combinatorics, i.e., of
the form $\Delta$, $\Delta \sqcup \{ 0 \}$, $\weight V$, $\conv(\weight
V)$.

\begin{remark}\label{212-closed interpretation}
There is a natural combinatorial interpretation for ``212-closed subsets
$Y \subseteq X$'' of a real vector space $\mathbb{V}$ when $X$ is the set
of lattice points in a lattice polytope -- i.e., $X$ is the intersection
of $\conv(X)$ with a lattice in $\mathbb{V}$ that contains the vertices
of $\conv(X)$. (We explain the connection to $X = \weight L(\lambda)$ in
Remark \ref{lattice polytope discussion}.) Suppose $Y$ denotes a subset
of ``colored'' (or in a contemporary spirit, ``infected'') lattice points
in $X$, with the property that if $y \in Y$ is the average of two points
in $X$, then the ``color'' or ``infection'' spreads to both points
from~$y$. (More precisely, if two pairs of -- not necessarily distinct --
points have the same average, and one pair is colored, then the color
spreads to the other pair.) We would like to understand the extent to
which the spread happens. A ``continuous'' variant works with the entire
convex hull itself, rather than the lattice points in $X$.
\end{remark}

Khare \cite{KhareJ.Alg} classified the weak faces and 212-closed subsets for $X=\weight V$, for $V$ from a large class of highest weight $\mathfrak{g}$-modules including all simple highest weight modules.
He showed that these two classes of subsets are equal, and they coincide with the sets of weights falling on the faces of the convex hull of weights. The interesting part here is:

\begin{remark}\label{lattice polytope discussion}
$\conv (\weight L(\lambda))$ is a convex polytope, and moreover,
a lattice polytope as in Remark~\ref{212-closed interpretation}; see
\eqref{Bump's question}, which explains the latter part. Khare's
(partial) results and ours below show that for \textit{all} such lattice
polytopes $\conv (\weight L(\lambda))$, the weak faces and 212-closed
subsets of $X$ are the same, and these two classes of subsets are
equivalent to the faces of $\conv(X)$. This equivalence is striking in
view of the ``minimality'' in the definition of 212-closed subsets,
particularly in contrast to the definition of (weak) faces. Furthermore,
it naturally raises the question of exploring for general lattice
polytopes, the extent to which these results hold; particularly, the
equivalence of these three notions.
\end{remark}

Recently in \cite{WFHMRS}, we have shown all of the (partial) results of
Khare \cite{KhareJ.Alg}, for important sets in Lie theory: $X= \Delta$ or
$\Delta\sqcup\{0\}=\weight \mathfrak{g}$, or $\weight V\text{ or
}\conv(\weight V)$ (the $\mathbb{R}$-\textit{convex hull} of $\weight V$)
for any highest weight module $V$, over any Kac--Moody $\mathfrak{g}$.
More precisely, for $X=\weight V$ (or $\conv (\weight V)$), we completely
classify these two classes of subsets of $X$ over any Kac--Moody algebra
$\mathfrak{g}$. More strongly, we show their equivalence with the sets of
weights on the faces (respectively, with the faces) of $\conv (\weight
V)$. For $X=\Delta\sqcup\{0\}$ or $X=\Delta$, we show the analogous
results, and that these two classes of subsets of $X$ are equal, except
in two interesting cases where $X=\Delta$ is of type either $A_2$ or
$\widehat{A_2}$. In this note, we discuss these results of \cite{WFHMRS}
for $\mathfrak{g}$ of finite and affine types. 

\section{Parabolic partial sum property}
Throughout, $\Delta$ is the root system of $\mathfrak{g}$ which is either
finite-dimensional simple or an affine Kac--Moody Lie algebra;
accordingly, we say $\Delta$ is of finite type or affine type,
respectively. $\Delta^+$ denotes the set of positive roots in $\Delta$.
Recall $\mathfrak{g}$ has the root space decomposition
$\mathfrak{h}\oplus \bigoplus_{\beta\in\Delta}\mathfrak{g}_{\beta}$,
where $\mathfrak{g}_{\beta} :=\{x\in\mathfrak{g}\ | \ h x=
\langle\beta,h\rangle x\ \forall \ h\in\mathfrak{h}\}$ is the root space
corresponding to $\beta$; here $\langle\beta,h\rangle$ denotes the
evaluation of $\beta\in\mathfrak{h}^*$ at $h\in\mathfrak{h}$. We begin by
developing some notation needed to state the parabolic partial sum
property.
Let $\mathcal{I}$ index the simple roots in $\Delta^+ \subset
\Delta$.
For $\emptyset\neq I\subseteq \mathcal{I}$ we define a special (for this entire paper) subset $\Delta_{I,1}$ of positive roots, and for a vector $x=\sum_{i\in\mathcal{I}} c_i\alpha_i$ for $c_i\in\mathbb{C}$ we define two important height functions, as follows: 
\begin{equation}\label{unit I-height roots}
  \height(x):=\sum_{i\in\mathcal{I}}c_i, \quad\text{and}\quad
  \height_I(x):=\sum_{i\in I} c_i.\qquad\quad
  \Delta_{I,1}:=\left\{\beta\in\Delta\ |\ \height_I(\beta)=1\right\}\
  \subseteq \Delta^+. \end{equation}
  The subsets $\Delta_{I,1}$ form the minimal generating sets for the cones in weights of simple highest weight modules mentioned in the introduction:

\begin{theorem}[{\bf Parabolic Partial Sum Property}]\label{P-psp}
Let $\Delta$ be a finite or an affine root system, and fix $\emptyset\neq
I\subseteq \mathcal{I}$. Suppose $\beta$ is a positive root with
$m=\height_I(\beta)>0$. Then
\[
\text{there exist roots }\ \ \gamma_1,\ldots,\gamma_m\in\Delta_{I,1}\ \
\text{ such that }\ \ \beta=\sum_{j=1}^m\gamma_j\ \ \text{ and }\ \
\sum_{j=1}^i\gamma_j\in\Delta^+\ \forall \ 1\leq i\leq m.
\]
In other words, every root with positive $I$-height is an ordered sum of
roots, each with unit $I$-height, such that each partial sum of that
ordered sum is also a root.
\end{theorem}

When $I=\mathcal{I}$, this is precisely the usual PSP. We
next present another example:

\begin{example}\label{type A_6 example}
 Let $\mathfrak{g}$ be of type $A_6$ ($7\times 7$ trace zero matrices
 over $\mathbb{C}$), $\mathcal{I}=\{1,2,3,4,5,6\}$, and fix $I=\{2,4,5\}$. The Dynkin diagram for $\mathfrak{g}$ (with nodes from $I$ boxed) is:   
\begin{center}
  \begin{dynkinDiagram}[
*/.append style={ultra thick, blue color},edge length=0.5cm,edge/.style={black,very thick}, labels={\textcolor{black}{\textbf{1}},\boxed{\textcolor{black}{\textbf{2}}},\textcolor{black}{\textbf{3}},\boxed{\textcolor{black}{\textbf{4}}},\boxed{\textcolor{black}{\textbf{5}}},\textcolor{black}{\textbf{6}}},scale=2]{A}{******}
  \end{dynkinDiagram}
\end{center}
 Recall, the roots in $\Delta$ are precisely $\sum_{i\in T}\alpha_i$,
 where $T \subset \mathcal{I}$ has consecutive indices.
 Now,
\[
\Delta_{I,1}=\big\{\alpha_2,\ \ \alpha_1+\alpha_2,\ \ \alpha_2+\alpha_3,\
\ \alpha_1+\alpha_2+\alpha_3,\ \
\alpha_4,\ \ \alpha_3+\alpha_4,\ \ \alpha_5,\ \ \alpha_5+\alpha_6\big\}.
\]
  Let $\beta=\sum_{i=1}^6\alpha_i$ denote the highest root in $\Delta$.
  Check that $\height(\beta)=6$ and $\height_I(\beta)=3$. Set
  $\gamma_1=\alpha_1+\alpha_2$, $\gamma_2=\alpha_3+\alpha_4$ and
  $\gamma_3=\alpha_5+\alpha_6$. Observe that
  $\gamma_1+\gamma_2+\gamma_3=\beta$, and moreover, both $\gamma_1$ and
  $\gamma_1+\gamma_2$ are roots.
 \end{example}

\begin{proof}[\textnormal{\textbf{Sketch of proof for parabolic PSP}}]
Fix $\emptyset\neq I\subseteq \mathcal{I}$ and a root $\beta$ with
$m=\height_I(\beta)>0$. The parabolic PSP trivially holds for $\beta$ if
$\height_I(\beta)=1$; hence, we now assume $\height_I(\beta)>1$. The
parabolic PSP is shown by applying the following stronger, structural
result:

\begin{theorem}\label{Lie words theorem}
Let $\mathfrak{g}$ be of finite or affine type. Fix $\emptyset\neq
I\subseteq \mathcal{I}$, and $\beta\in \Delta^+$ with
$m := \height_I(\beta)>0$. Then the root space $\mathfrak{g}_{\beta}$ is
spanned by the right normed Lie words of the form:
\begin{equation}\label{right normed Lie words}
\big[e_{\gamma_m},[\cdots, [e_{\gamma_2},e_{\gamma_1}]\cdots]\big]\
\text{ where }\gamma_t\in\Delta_{I,1},\
e_{\gamma_t}\in\mathfrak{g}_{\gamma_t}, \ \forall\ 1\leq t\leq m,\
\text{and }\sum_{t=1}^m \gamma_t=\beta.
\end{equation}
\end{theorem}

In fact, a stronger result (in \cite{MBWWM}) shows the parabolic PSP to
the best possible extent, and moreover at the level of \textit{Lie
words}, via proving Theorem \ref{Lie words theorem} for \textit{any}
general Lie algebra (over any field) graded over \textit{any} free
abelian semigroup. 
Indeed, by (the stronger analogue of) Theorem \ref{Lie words theorem},
let the Lie word in \eqref{right normed Lie words} be non-zero. Then all
its right normed Lie sub-words -- namely, $\big[e_{\gamma_i},[\cdots,
[e_{\gamma_2},e_{\gamma_1}]\cdots\big]$ $\forall$ $1\leq i \leq m$ -- are
non-zero. This immediately implies that each partial sum of
$\sum_{t=1}^m\gamma_t$ is a root, proving the parabolic PSP.

Finally, the proof for Theorem~\ref{Lie words theorem} involves structure
theory ideas, and goes as follows. Fix a \textit{right normed Lie word}
$0\neq x=\big[e_{i_n},[\cdots,[e_{i_2},e_{i_1}]\cdots]\big]$ in
$\mathfrak{g}_{\beta}$, where $e_{i_t}$ is a (simple) root vector in the
simple root space $\mathfrak{g}_{\alpha_{i_t}}$ $\forall$ $1\leq t\leq n$
and $\sum_{t=1}^n\alpha_{i_t}=\beta$. Recall, every root space of
$\mathfrak{g}$ is spanned by such right normed Lie words. Inducting on
$n$, we can show that $x$ can be written as a linear combination of the
Lie words as in \eqref{right normed Lie words}. More precisely, we assume
that $\big[e_{i_{n-1}},[\cdots,[e_{i_2},e_{i_1}]\cdots\big]$ is in the
span of Lie words as in \eqref{right normed Lie words}, and then use the
Jacobi identity to take $e_{i_n}$ inside these new Lie words if
$\height_I(\alpha_{i_n})=0$. 
 \end{proof}

Alternately for $\mathfrak{g}$ of finite type, the parabolic PSP is
implied by:

\begin{lemma}
Assume that $\mathfrak{g}$ is of finite type, and $\emptyset\neq
I\subseteq \mathcal{I}$. Suppose $\beta\in\Delta^+$ is such that
$\height_I(\beta)>1$. Then there exists a root $\gamma\in\Delta_{I,1}$
such that $\beta-\gamma\in\Delta^+$.  
\end{lemma}

\begin{proof}
We reach a contradiction, working with
$\beta$ of least height such that the lemma fails
for $\beta$. Recall, $\beta$ is a sum of
simple roots, and the Killing form on $\mathfrak{h}$, hence on
$\mathfrak{h}^*$, is positive definite, i.e.\ $(x,x)>0$ $\forall$
$x\in\mathfrak{h}^*$. Thus, fix a simple root $\alpha$ such that
$(\beta,\alpha)>0$. Now \cite[Lemma 9.4]{Humphreys} implies
$\beta-\alpha$ is a root. The assumption on $\beta$ forces
$\height_I(\alpha)=0$, so $\height_I(\beta-\alpha)=\height_I(\beta)>1$.
Since $\height(\beta-\alpha)<\height(\beta)$, we have a root
$\eta\in\Delta_{I,1}$ such that $\beta-\alpha-\eta\in\Delta^+$. Now
applying \cite[Lemma 1.1(iii)]{Dolbin} for $\beta-\alpha-\eta$, $\eta$
$\alpha$ (in place of $\alpha,\beta$, $\gamma$, respectively), we have
either: (1)~$(\beta-\alpha-\eta)+(\alpha)=\beta-\eta$ is a root, or (2)
$\alpha+\eta$ is root \big(in which case
$\alpha+\eta\in\Delta_{I,1}$\big). Both of these cases contradict the
choice of $\beta$.
\end{proof}

\section{Minkowski difference formulas for weights}\label{S3}

In this section, we state and discuss our minimal description result for
$\weight L(\lambda)$ for all weights $\lambda$ $\in\mathfrak{h}^*$, and a
Minkowski difference formula for the set of weights of an arbitrary
highest weight $\mathfrak{g}$-module, using the parabolic PSP. We need
the following notation.

For any subset $S\neq \emptyset$ of a real vector space, let
$\mathbb{Z}S$ (or $\mathbb{Z}_{\geq 0}S$) comprise the set of
$\mathbb{Z}$-linear (or $\mathbb{Z}_{\geq 0}$-linear) combinations of
elements of $S$. Let $\conv (S)$ denote the $\mathbb{R}$-convex hull of
$S$. Recall, $\mathcal{I}$ is the set of nodes in the Dynkin diagram of
$\mathfrak{g}$. Set $I^c := \mathcal{I}\setminus I$ for $I\subseteq
\mathcal{I}$. Now for $i \in \mathcal{I}$, let $s_i$ denote the simple
reflection about the hyperplane perpendicular to the simple root
$\alpha_i$. These generate the Weyl group $W$ of $\mathfrak{g}$. Let
$e_i, f_i,\alpha_i^{\vee}$ $\forall$ $i\in\mathcal{I}$ be the Chevalley
generators of $\mathfrak{g}$.

We now fix $\emptyset\neq J\subseteq \mathcal{I}$ and define the
parabolic analogues. Define $\mathfrak{g}_J$ to be the Lie subalgebra of
$\mathfrak{g}$ generated by $e_j,f_j, \alpha_j^{\vee}$ $\forall$ $j\in
J$. Since $\mathfrak{g}$ is of finite or affine type, recall that
$\mathfrak{g}_J$ is always semisimple for $J \subsetneq \mathcal{I}$. The
subalgebra $\mathfrak{g}_J$ corresponds to the Cartan matrix $A_{J\times
J}$. We denote the Cartan subalgebra, root system, the (fixed) simple
roots and simple co-roots of $\mathfrak{g}_J$ by $\mathfrak{h}_J$,
$\Delta_J$, $\Pi_J:=\{\alpha_j\}_{j\in J}$ and $\Pi_J^{\vee}$,
respectively; note that $\Delta_J=\Delta\cap \mathbb{Z}\Pi_J$. The
parabolic subgroup of $W$ generated by the simple reflections
$\{s_j\}_{j\in J}$ corresponding to $J$, is the Weyl group of
$\mathfrak{g}_J$; it is denoted by $W_J$.  

Fix $\lambda\in\mathfrak{h}^*$ for this section. Recall, $\lambda$ is \textit{integral}, respectively \textit{dominant}, if $\langle\lambda,\alpha_i^{\vee}\rangle\in\mathbb{Z}$, respectively $\langle\lambda,\alpha_i^{\vee}\rangle\in\mathbb{R}_{\geq 0}$, for all $i\in\mathcal{I}$; $\lambda$ is \textit{dominant integral} if it is both. For a $\mathfrak{g}$-module $V$, recall the definitions of the $\lambda$-weight space and the set of weights of $V$:
\[ 
V_{\lambda}:=\{v\in M\text{ }|\text{ }h\cdot v= \langle\lambda,h\rangle v\text{ }\forall\text{ }h\in \mathfrak{h}\}\quad\text{and}\quad \weight V:=\{\mu\in\mathfrak{h}^*\text{ }|\text{ }V_{\mu}\neq\{0\}\}.
\]

$V$ is said to be a highest weight module of highest weight $\lambda$, if
there exists a vector $0\neq v\in V$ such that (1) $v$ generates $V$ over
$\mathfrak{g}$, (2) $hv=\langle\lambda,h\rangle v$ $\forall$
$h\in\mathfrak{h}$ (so $v\in V_{\lambda}\neq\{0\}$), and (3) $e_iv=0$
$\forall$ $i\in\mathcal{I}$. We call such a $v$ to be a highest weight
vector of $V$ and it is unique up to scalars. We denote the simple
highest weight module over $\mathfrak{g}_J$ with highest weight $\lambda$
\big(or rather $\lambda$ restricted to $\mathfrak{h}_J$\big) by
$L_J(\lambda)$. In this section, we deal with an important subset $
J_{\lambda}\ :=\ \{j\in \mathcal{I}\ |\
\langle\lambda,\alpha_j^{\vee}\rangle\in \mathbb{Z}_{\geq 0} \}$ of
$\mathcal{I}$ 
-- the \textit{integrability} of $\lambda$ (or of
$L(\lambda)$).

\begin{remark}
$L_{J_{\lambda}}(\lambda)$ is the integrable highest weight
$\mathfrak{g}_{J_{\lambda}}$-module corresponding to $\lambda$; i.e.,
$f_j$ acts nilpotently on each vector in $L_{J_{\lambda}}(\lambda)$
(nilpotently on all of $L_{J_{\lambda}}(\lambda)$ when $\mathfrak{g}$ is
of finite type) $\forall$ $j\in J_{\lambda}$. So, we know well the set
$\weight L_{J_{\lambda}}(\lambda)$. It (or equivalently, the set of nodes
$J_{\lambda}$) determines $\weight L(\lambda)$, as shown by Khare
\cite{KhareJ.Alg} and Dhillon--Khare \cite{KhareAdv}:
\begin{align}\label{Minkowski difference formula for wt L(lambda)}
    \weight L(\lambda)\ =\ & \ \weight L_{J_{\lambda}}(\lambda)\ -\ \mathbb{Z}_{\geq 0}\left(\Delta^+\setminus \Delta_{J_{\lambda}}^+\right).\\
\weight L(\lambda) \ = \ & \ \conv\left(\weight L(\lambda)\right) \cap \left(\lambda-\mathbb{Z}_{\geq 0}\Pi\right).\label{Bump's question}
\end{align}
\end{remark}

\noindent Observe by formula \eqref{Minkowski difference formula for wt
L(lambda)} that $\conv (\weight L(\lambda))$ equals the Minkowski
difference between $\conv\big(\weight L_{J_{\lambda}}(\lambda)\big)$ and
the real cone $\mathbb{R}_{\geq 0}\big(\Delta^+\setminus
\Delta^+_{J_{\lambda}}\big)$. Formula \eqref{Bump's question}
affirmatively answers a question of Daniel Bump on recovering weights
from their convex hulls for these simple $\mathfrak{g}$-modules. Formula
\eqref{Minkowski difference formula for wt L(lambda)} is the Minkowski
difference formula mentioned in the introduction, expressing $\weight
L(\lambda)$ in terms of two well-understood subsets. The parabolic
partial sum property was posed in order to determine the minimal
generators for the (non-negative integer) cone in \eqref{Minkowski
difference formula for wt L(lambda)}, which is generated by all the
positive roots outside $\Delta^+_{J_{\lambda}}$; these are found in the
next theorem. Recall the definition of $\Delta_{J_{\lambda}^c, \ 1}$ from
\eqref{unit I-height roots}. 

\begin{theorem}
The cone $\mathbb{Z}_{\geq 0}
\left(\Delta^+\setminus\Delta_{J_{\lambda}}^+\right)$ is minimally
generated (over $\mathbb{Z}_{\geq 0}$) by $\Delta_{J_{\lambda}^c, \ 1}$,
which is always finite if $\Delta$ is of finite or affine type.
Therefore, we have the following minimal description:
\begin{equation}
\weight L(\lambda)\ =\ \weight L_{J_{\lambda}}(\lambda)-\mathbb{Z}_{\geq
0}\Delta_{J_{\lambda}^c,\ 1}, \qquad \forall \lambda \in \mathfrak{h}^*.
\end{equation}
\end{theorem}

\begin{proof}[Sketch of proof]
Notice that $\Delta_{J_{\lambda}^c,\ 1}\subseteq \Delta^+\setminus
\Delta_{J_{\lambda}}^+$. Conversely, if $\beta\in
\Delta^+\setminus\Delta^+_{J_{\lambda}}$, then
$\height_{J_{\lambda}^c}(\beta)>0$. Now the partial sum property implies
$\beta\in \mathbb{Z}_{\geq 0}\Delta_{J_{\lambda}^c,\ 1}$. Therefore,
$\mathbb{Z}_{\geq 0}\left(\Delta^+\setminus
\Delta_{J_{\lambda}}^+\right)=\mathbb{Z}_{\geq 0}\Delta_{J_{\lambda}^c, \
1}$. It remains to show the minimality -- or irredundancy --
of $\Delta_{J_{\lambda}^c,\ 1}$. Suppose there is a root
$\gamma\in\Delta_{J_{\lambda}^c,\ 1}$ such that
$\gamma=\sum_{i=1}^n\gamma_i$ for some roots $\gamma_1,\ldots,\gamma_n\in
\Delta^+\setminus \Delta_{J_{\lambda}}^+$. Comparing the
$J_{\lambda}^c$-heights on both sides, it follows that $n=1$ and
$\gamma_1=\gamma$.
\end{proof}

Next, we give our formula for the weights of an arbitrary highest weight
$\mathfrak{g}$-module $V$ of highest weight $\lambda\in\mathfrak{h}^*$.
We define for convenience $\weight_{J_{\lambda}} V:=
(\lambda-\mathbb{Z}_{\geq 0}\Pi_{J_{\lambda}})\cap \weight V$, which is
the set of weights of the $\mathfrak{g}_{J_{\lambda}}$-module generated
by a highest weight vector in $V$.

\begin{theorem}\label{wt V formula theorem}
Let $\lambda\in\mathfrak{h}^*$, and $V$ be a highest weight $\mathfrak{g}$-module with highest weight $\lambda$. Then
\begin{equation}\label{Minkowski difference formula for wt V}
\weight V\ =\ \weight _{J_{\lambda}}V-\mathbb{Z}_{\geq 0}\Delta_{J_{\lambda}^c,\ 1}\ =\ \weight _{J_{\lambda}}V-\mathbb{Z}_{\geq 0}(\Delta^+\setminus \Delta_{J_{\lambda}}^+).
\end{equation}
\end{theorem}

\begin{remark}
When $V=L(\lambda)$, \eqref{Minkowski difference formula for wt V} is the same as the formula \eqref{Minkowski difference formula for wt L(lambda)} by Dhillon and Khare, as $\weight_{J_{\lambda}}L(\lambda)=\weight L_{J_{\lambda}}(\lambda)$.
The significance of the formula \eqref{Minkowski difference formula for wt V} is as follows:

\noindent The cone in it is well understood, and the unknown part is
$\weight_{J_{\lambda}}V$. 
In view of this, if we find the sets of weights of highest weight $\mathfrak{g}$-modules with dominant integral highest weights, then we have the formulas for weights of all highest weight $\mathfrak{g}$-modules. 
\end{remark}

\begin{proof}[Sketch of proof of Theorem \ref{wt V formula theorem}]
The inclusion $\weight V\subseteq \weight
_{J_{\lambda}}V-\mathbb{Z}_{\geq 0}\Delta_{J_{\lambda}^c,\ 1}$ follows by
\cite[Corollary 3.2(b)]{MBWWM}, which involves the
Poincar\'{e}--Birkhoff--Witt theorem and the parabolic PSP.
To show the reverse inclusion in the previous
sentence, recall that \cite[Lemma 4.2]{MBWWM} says $\weight
_{J_{\lambda}}V-\mathbb{Z}_{\geq 0}\Pi_{J_{\lambda}^c}\subseteq \weight
V$. Using this inclusion as the base step, for any
$\gamma_1,\ldots,\gamma_n\in\Delta_{J_{\lambda}^c,\ 1}$, we induct on
$\height_{ J_{\lambda}^c}\left(\sum_{i=1}^n\gamma_i\right)$ to prove:
$\weight_{J_{\lambda}}V-\mathbb{Z}_{\geq 0}\sum_{i=1}^n\gamma_i\
\subseteq \weight V$, which finishes the proof. Showing this crucial step
is tremendously simplified if one works with the generating
set $\Delta_{J_{\lambda}^c,\ 1}$ for $\mathbb{Z}_{\geq
0}\left(\Delta^+\setminus \Delta^+_{J_{\lambda}}\right)$ instead of
$\Delta^+\setminus \Delta^+_{J_{\lambda}}$. 
\end{proof}

\section{Weak faces and 212-closed subsets}

Recall the definitions of weak faces and 212-closed subsets from
Definition \ref{weak face defiinition}. Throughout, $\mathfrak{g}$ (or
equivalently, $\Delta$) is of finite or affine type, $\mathbb{A}\neq
\{0\}$ is a fixed arbitrary additive subgroup of $\mathbb{R}$, and $V$ is
an arbitrary highest weight $\mathfrak{g}$-module of highest weight
$\lambda\in\mathfrak{h}^*$. The symbol $Y$ always denotes a weak face or
a 212-closed subset of $X$, where $X=\Delta$ or $\Delta\sqcup\{0\}$ or
$\weight V$ or $\conv(\weight V)$. The main results of this section are
as follows. Theorem \ref{weak faces of roots main theorem} completely
determines as well as discusses the equivalence of the 212-closed subsets
and weak faces both for $X=\Delta$ and $X=\Delta\sqcup\{0\}$. It turns
out that $X=\Delta$ of types $A_2$ or (affine) $\widehat{A_2}$ are the
only two exceptions for which these two classes of subsets are not equal.
In Theorem \ref{weak faces of weights main theorem}: 1) we show for
$X=\weight V$ that these two classes of subsets are the same as the sets
of weights falling on the faces of $\conv(\weight V)$, and 2) for
$X=\conv(\weight V)$ these two notions are the same as the usual faces.
Let us first understand a geometric interpretation of 212-closedness.

\begin{remark}\label{geometric interpretation of 212-closedness}
In the definition of a 212-closed set, consider the equality $(y_1) +
(y_2) = (x_1) + (x_2)$ for $y_1,y_2\in Y$ and $x_1,x_2\in X$ implying
$x_1,x_2\in Y$. This equality arises notably when: (1) $y_1=y_2$ is the
midpoint between $x_1$ and $x_2$. (2) $x_1=x_2$ is the midpoint between
$y_1$ and $y_2$. (3) $y_1, y_2, x_1, x_2$ form vertices of a
parallelogram with $y_1, y_2$ (similarly, $x_1, x_2$) on a diagonal.
\end{remark}

Using this remark, we now look at the 212-closed subsets of $\Delta$ for
the example below:

\begin{example}\label{212-closed subsets in type A_2 example}
Let $\Delta$ be of type $A_2$, with the simple roots $\alpha_1$ and
$\alpha_2$. $\Delta=\{\pm \alpha_1,\ \pm\alpha_2$,
$\pm(\alpha_1+\alpha_2) \}$ is the set of vertices of the hexagon in
Figure \eqref{A_2 root system} below, where $\alpha_1=(1,0)$ and
$\alpha_2=(-1/2, \sqrt{3}/2)$. The figure in $\eqref{A_2 root polytope}$
is the $\mathbb{R}$-convex hull of $\Delta$.
\begin{figure}[h]
\centering
\begin{subfigure}{.4\linewidth}
\begin{tikzpicture}[scale=1.5]
\draw
(cos 0, sin 0) --
(cos 60, sin 60) --
(cos 120, sin 120) --
(cos 180, sin 180) --
(cos 240, sin 240) --
(cos 300, sin 300) --
cycle;
\filldraw[black](cos 0,sin 0) circle (2pt) node[anchor=west]{$(1,0)$};
\filldraw[black](cos 60,sin 60) circle (2pt) node[anchor=west]{$\left(\frac{1}{2},\frac{\sqrt{3}}{2}\right)$};
\draw (0,0) node{(0,0)};
\filldraw[black](cos 120,sin 120) circle (2pt) node[anchor=east]{$\left(-\frac{1}{2},\frac{\sqrt{3}}{2}\right)$\hspace{3pt}};
\filldraw[black](cos 180,sin 180) circle (2pt) node[anchor=east]{$(-1,0)$};
\filldraw[black](cos 240,sin 240) circle (2pt) node[anchor=east]{$\left(-\frac{1}{2},-\frac{\sqrt{3}}{2}\right)$};
\filldraw[black](cos 300,sin 300) circle (2pt) node[anchor=west]{$\left(\frac{1}{2},-\frac{\sqrt{3}}{2}\right)$};
\end{tikzpicture}
\caption{Type $A_2$ root system}\label{A_2 root system}
\end{subfigure}
\begin{subfigure}{.4\linewidth}
\begin{tikzpicture}[scale=1.5]
\filldraw[fill=blue!70]
(cos 0, sin 0) --
(cos 60, sin 60) --
(cos 120, sin 120) --
(cos 180, sin 180) --
(cos 240, sin 240) --
(cos 300, sin 300) --
cycle;
\filldraw[black](cos 0,sin 0) circle (2pt) node[anchor=west]{$\alpha_1$};
\filldraw[black](cos 60,sin 60) circle (2pt) node[anchor=west]{$\alpha_1+\alpha_2$};
\draw (0,0) node{(0,0)};
\filldraw[black](cos 120,sin 120) circle (2pt) node[anchor=east]{$\alpha_2$};
\filldraw[black](cos 180,sin 180) circle (2pt) node[anchor=east]{$-\alpha_1$};
\filldraw[black](cos 240,sin 240) circle (2pt) node[anchor=east]{$-\alpha_1-\alpha_2$};
\filldraw[black](cos 300,sin 300) circle (2pt) node[anchor=west]{$-\alpha_2$};
\end{tikzpicture}
\caption{Type $A_2$ root polytope}\label{A_2 root polytope}
\end{subfigure}
\caption{The $A_2$ case}
\end{figure}
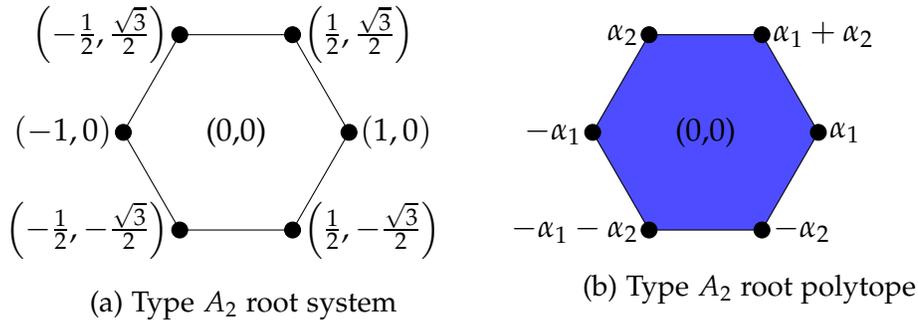

$\Delta$ is 212-closed in itself trivially. By Remark \ref{geometric
interpretation of 212-closedness}(1), since no root lies between
two others, all singletons in $\Delta$ are 212-closed.
By Remark \ref{geometric
interpretation of 212-closedness} the following are all the rest of the 212-closed subsets in $\Delta$:
(i) non-antipodal pairs of points;
(ii) any subset of 3 vertices, such that in their induced subgraph none
of them are isolated, or all of them are isolated.
Next about the 212-closed subsets of $\Delta\sqcup\{0\}$, once again by
Remark \ref{geometric
interpretation of 212-closedness} points (1) and (3), if $Y$ 212-closed in $\Delta\sqcup\{0\}$ is such that $Y$ contains $0$ or a
pair of non-adjacent vertices, then $Y=\Delta\sqcup\{0\}$. 
Theorem \ref{weak faces of weights main theorem}(b), and the formulas
given for the faces and for weights falling on the faces \eqref{standard
face}, can be verified/understood with the aid of the convex set in
Figure \eqref{A_2 root polytope}, which equals
$\conv(\weight L(\alpha_1+\alpha_2))$.      

\end{example}
We now generalize below some of the notable observations in this example.
\begin{note}\label{all 212-closed sets are proper}
 For $X=\Delta$ or $\Delta\sqcup\{0\}$, throughout, every weak face and
 every 212-closed subset $Y$ of $X$ will be assumed to be a proper subset
 of $X$. This is to overcome the obstruction to our uniform comparisons
 of these notions: $\Delta$ is 212-closed (also a weak-$\mathbb{A}$-face
 for any $\mathbb{A}$) in $\Delta$, but not in $\Delta\sqcup\{0\}$, since
 if $Y=\Delta$ is 212-closed in $\Delta\sqcup\{0\}$, then
\begin{equation}\label{Delta is not 212-closed in Delta U 0}
    \text{for any }\xi\in Y=\Delta,\quad (\xi)+(-\xi)=2(0)\quad \implies 0\in \Delta \ \Rightarrow\!\Leftarrow. 
\end{equation}
\end{note}

\begin{remark}\label{0 not in 212-closed subsets}
Let $Y$ be 212-closed in $X=\Delta$ or $\Delta\sqcup\{0\}$. Since
$Y\subsetneqq X$, observe via the reversed equation before the
implication in \eqref{Delta is not 212-closed in Delta U 0} that $0\notin
Y$. This has two further implications:
(1) For any root $\xi$, both $\xi$ and $-\xi$ cannot belong to $Y$.
(2) Every 212-closed subset of $\Delta\sqcup\{0\}$ is contained in
$\Delta$, and therefore (by the definition) it is 212-closed in $\Delta$.
These assertions hold equally well for weak faces, since they are
212-closed.  
\end{remark}

Let $\mathfrak{g}$ be of finite type, with highest long root $\theta$. As
$\mathfrak{g}$ and $L(\theta)$ are isomorphic,~$\weight \mathfrak{g}$ =
$\Delta\sqcup\{0\}=\weight L(\theta)$. So \cite[Theorem C]{KhareJ.Alg},
which finds and shows the equivalence of all the weak faces and
212-closed subsets of $X=\weight L(\lambda)$ $\forall$
$\lambda\in\mathfrak{h}^*$, applied for $\lambda=\theta$~yields:

\begin{proposition}
For $\mathfrak{g}$ of finite type, the following classes of subsets are
all the same: 1) 212-closed subsets of $\Delta\sqcup\{0\}$, 2) weak faces
of $\Delta\sqcup\{0\}$, 3) the sets of weights falling on the faces of
$\conv(\Delta\sqcup\{0\})$, 4) maximizer subsets for linear functionals
on $\Delta\sqcup\{0\}$. These are precisely: 
\begin{equation}\label{standard face}
w\left[\big(\theta-\mathbb{Z}_{\geq 0}\Pi_I\big)\cap \weight
L(\theta)\right]\quad\quad \text{for all }w\in W\text{ and }I\subsetneqq
\mathcal{I}.
\end{equation}
\end{proposition}

\begin{remark}
Recall, the convex hulls of the subsets in \eqref{standard face} are all
the faces of the root polytope $\conv(\Delta)$ from Borel--Tits
\cite{Borel}, Satake \cite{Satake}, Vinberg \cite{Vinberg}; for instance
see \cite[Theorem 2.17]{KhareJ.Alg} which quotes this result. Let
$\omega_i$ $\forall$ $i\in\mathcal{I}$ be the fundamental dominant
integral weights in $\mathfrak{h}^*$. The set in \eqref{standard face}
maximizes the linear functional $\left(w\sum_{j\in I^c} \omega_j\
,-\right)$. So by Remark \ref{weak face is 212-closed}, it is a
weak-$\mathbb{A}$-face (also 212-closed) in $\Delta\sqcup\{0\}$ for every
$\emptyset\neq \mathbb{A}\subseteq (\mathbb{R},+)$.
\end{remark}

Now assume that $\Delta$ is of affine type and
$\mathcal{I}=\{0,1,\ldots,\ell\}$, where $\ell$ is the rank of $A$ or
$\Delta$. A root $\beta$ is real if $(\beta,\beta)>0$, and imaginary if
$(\beta,\beta)=0$; every root is either real or imaginary. $W\Pi$ (the
$W$ orbit of all the simple roots) are all the real roots. The imaginary
roots are $\{n\delta\ |\ n\in\mathbb{Z}\setminus\{0\}\}$, where $\delta$
is the smallest positive imaginary root. In a 212-closed subset of
$\Delta$, every root is  real by the remark below. 

\begin{remark}
Let $Y$ be 212-closed in $\Delta$ of affine type and $\eta\in\Delta$
imaginary. Then $\eta\notin Y$:
\begin{equation}
\text{If }\eta\in Y,\quad \text{then }\ 2(\eta)=(3\eta)+(-\eta) \ \text{
implies }\ \pm \eta\in Y \ \  \Rightarrow\!\Leftarrow\ \text{(by Remark
\ref{0 not in 212-closed subsets})}. 
\end{equation}
\end{remark}

Recall from Tables Aff 1--Aff 3 and Subsection 6.3 of Kac's
book~\cite{Kac}: (1)~the subroot system $\mathring{\Delta}$ generated by
$\{\alpha_1,\ldots,\alpha_{\ell}\}$ is of finite type; (2) the roots in
$\Delta$ are explicitly describable in terms of the roots in
$\mathring{\Delta}$. Recall, there are at most two lengths in a finite
type root system. Let $\mathring{\Delta}_s$ and $\mathring{\Delta}_l$
denote the set of roots in $\mathring{\Delta}$ of the shortest and
longest lengths, respectively. Note,
$\mathring{\Delta}_s=\mathring{\Delta}_l$ when $\mathring{\Delta}$ is of
(simply laced) types $A_n,D_n,E_6,E_7,E_8$.
With the above preliminaries and observations, we are ready to state our
next main theorem, with one last notation:

For $\Delta$ of type $X_{\ell}^{(r)}$, $r\in\{1,2,3\}$ -- see
\cite[Tables Aff 1--Aff 3]{Kac} -- and $Y\subseteq \Delta\sqcup\{0\}$, define
	\begin{equation}\label{E2.4}
	Y_s:=\begin{cases}
	(Y\cap\mathring{\Delta}_s)+\mathbb{Z}\delta &\text{if } Y\cap \mathring{\Delta}_s\neq\emptyset,\\
	\emptyset &\text{if }Y\cap\mathring{\Delta}_s=\emptyset,
	\end{cases}\quad  Y_l:=\begin{cases}
	(Y\cap\mathring{\Delta}_l)+r\mathbb{Z}\delta &\text{if } Y\cap \mathring{\Delta}_l\neq\emptyset,\\
	\emptyset &\text{if }Y\cap\mathring{\Delta}_l=\emptyset.
	\end{cases}
	\end{equation}   
\begin{theorem}\label{weak faces of roots main theorem}
	\begin{itemize}
		\item[(a)] For $\mathfrak{g} \neq A_2$ of finite type, the following classes of subsets are all the same: 212-closed subsets of $\Delta$, 212-closed subsets of $\Delta\sqcup\{0\}$, weak faces of $\Delta$, weak faces of $\Delta\sqcup\{0\}$, and the subsets in \eqref{standard face}.
		\item[(b)] For $\mathfrak{g}$ of type $A_2$, the following classes of subsets are all the same: 212-closed subsets of $\Delta\sqcup\{0\}$, weak faces of $\Delta$, weak faces of $\Delta\sqcup\{0\}$, and the subsets in \eqref{standard face}. All of these subsets and the following additional ones, are all the 212-closed subsets of $\Delta$.    \begin{equation}
\label{non standard 212-closed subsets}		
W\text{-conjugates of}:\ \Pi=\{\alpha_1,\alpha_2\},\
\Delta^+=\{\alpha_1,\alpha_2,\alpha_2+\alpha_1\},\
\{\alpha_1,\alpha_2,-\alpha_2-\alpha_1\}.\end{equation}
			\item[(c)] Assume that $\mathfrak{g}$ is of affine type. Then the following two statements 1) and 2) are equivalent: 
	\begin{itemize}
		\item[1)] $Y$ is a 212-closed subset of $\Delta$
		\big(respectively, of $\Delta\sqcup\{0\}$\big).
		\item[2)] $Y=Z_s\cup Z_l$ for some 212-closed subset $Z$
		of $\mathring{\Delta}$ \big(respectively, of
		$\mathring{\Delta}\sqcup\{0\}$\big). 
	\end{itemize}
	For $\mathfrak{g}$ not of type $\widehat{A_2}$ (so
	$\mathring{\Delta}$ is not of type $A_2$) -- similar to part
	(a) -- we have the equality of the four classes: 212-closed
	subsets and weak faces of both $\Delta$ and $\Delta\sqcup\{0\}$.
	When $\mathfrak{g}$ is of type $\widehat{A_2}$
	($\mathring{\Delta}$ is of type $A_2$) -- as in part (b) -- these
	classes other than the 212-closed subsets of $\Delta$ are the
	same. The additional 212-closed subsets of $\Delta$ correspond to
	those of $\mathring{\Delta}$ in \eqref{non standard 212-closed
	subsets}.
	\end{itemize}
\end{theorem}

\begin{proof}[Ideas in proof]
In \cite{WFHMRS}, each part of the theorem is shown in cases. Part (a)
for types $A_1$, $B_2$ and $G_2$, as well as part (b), can be verified
alternately via looking at their pictures/plots and the observations
similar to those in Example \ref{212-closed subsets in type A_2 example}.
In particular, in this manner, for $\Delta$ of type $A_2$, the
disjointedness of the two lists in \eqref{standard face} and \eqref{non
standard 212-closed subsets} can be verified. If $Y$ belongs to the list
in \eqref{non standard 212-closed subsets}, then $Y$ is not a weak face
of $\Delta$ of type $A_2$, because 
\[
\text{for any } \mathbb{A}\subseteq (\mathbb{R},+)\text{ and }0<
a\in\mathbb{A},\quad a(\alpha_1)+a(\alpha_2) = a(\alpha_1+\alpha_2)+(0)\
\implies 0\in Y,
\]
which is a contradiction. The proof in \cite{WFHMRS} of part (c), where
$\Delta$ is of affine type, runs over four steps. In all of them, we
heavily use the description of the roots given by the well-known result
\cite[Proposition 6.3]{Kac}. This description and the definition of
212-closed sets immediately give the equivalence of points 1) and 2) in
part (c). This close relation between the 212-closed subsets of $\Delta$
and $\mathring{\Delta}$ leads to the equivalences of the various classes
of sets in part (c).
\end{proof}

We conclude with our final main theorem for the set of weights and their
convex hull, for an arbitrary highest weight $\mathfrak{g}$-module $V$.
When $V=L(\lambda)$, part (a) of our theorem below recovers \cite[Theorem
C]{KhareJ.Alg} of Khare. We define the \textit{integrability} of $V$ to
be $I_V:=\{i\in \mathcal{I}\ |\ f_i \text{ acts nilpotently on each
vector in }V\text{ or equivalently on}\ V_{\lambda}\}$.
Recall, $V$ is an integrable $\mathfrak{g}_{I_V}$-module, and so $\weight
V$ and $\conv(\weight V)$ are $W_{I_V}$-invariant. $I_V$ determines
$\weight V$ to a significant extent, and $\conv(\weight V)$ completely,
by \cite{KhareJ.Alg} and \cite{KhareAdv}.

	\begin{theorem} \label{weak faces of weights main theorem}
	Let $V$ be an arbitrary highest weight $\mathfrak{g}$-module of
	highest weight $\lambda\in\mathfrak{h}^*$. Then:
	\begin{itemize}
	\item[(a)] The following classes of subsets are equal:
	(1)~Weak faces of $\weight V$.
	(2)~212-closed subsets of $\weight V$.
	(3)~The subsets
	\begin{equation}\label{standard faces of wt V}
	    w[(\lambda-\mathbb{Z}_{\geq 0}\Pi_I)\cap\weight V]\qquad  \text{for all }\  w \in W_{I_V}\ \text{ and }\ I\subseteq \mathcal{I}.
	\end{equation}
	\item[(b)] The following classes of subsets of $X =
	\conv_{\mathbb{R}} (\weight V)$ are equal:
	(1)~Exposed faces, i.e., the maximizer subsets of $X$ with
	respect to linear functionals.
	(2)~Weak faces of $X$.
	(3)~212-closed subsets of $X$.
	(4)~The convex hulls of the subsets in \eqref{standard faces of wt V}. 
   \end{itemize}
	\end{theorem}



\end{document}